\documentclass[12pt,reqno]{amsart}
\usepackage{amsthm,amssymb,amsbsy,epsfig,graphicx,color}
\usepackage{hyperref}
\usepackage[active]{srcltx}

\input{amssym.def}
\input{amssym}

\def\be{\begin{equation}}
\def\ee{\end{equation}}
\def\ba{\begin{eqnarray}}
\def\ea{\end{eqnarray}}
\def\lb{\label}
\def\nn{\nonumber}
\def\ni{\noindent}
\def\cal{\mathcal}

\makeatletter
\@addtoreset{equation}{section}\makeatother

\newcounter{theorem}

\newtheorem{prop}[theorem]{Proposition}

\newtheorem{theor}[theorem]{Theorem}

\begin{document}
\begin{titlepage}
\title[]
{Representations of finite-dimensional quotient algebras of the $3$-string braid group}

\author{Pavel Pyatov \&  Anastasia Trofimova}

\address{Pavel Pyatov:
National Research University Higher School of Economics
20 Myasnitskaya street, Moscow 101000, Russia \&
Bogoliubov Laboratory of Theoretical Physics, Joint Institute
for Nuclear Research, 141980 Dubna, Moscow Region, Russia}
\email{pyatov@theor.jinr.ru}

\address{Anastasia Trofimova:
National Research University Higher School of Economics
20 Myasnitskaya street, Moscow 101000, Russia \&
Center for Advanced Studies, Skolkovo Institute of Science and Technology,
Moscow, Russia}
\email{nasta.trofimova@gmail.com}

\thanks{
This research was carried out in frames of the HSE University Basic Research Program and
was supported jointly by the  Russian Academic Excellence Project '5-100' and by the
grant of RFBR No.16-01-00562 a.}


\begin{abstract}
We consider quotients of the group algebra of the $3$-string braid group $B_3$ by $p$-th order generic polynomial relations on the elementary braids. In cases $p=2,3,4,5$ these quotient algebras are finite dimensional. We give semisimplicity criteria for these algebras and present explicit formulas for all their irreducible representations.
\vspace{5mm}


\end{abstract}

\maketitle

\end{titlepage}

\section*{Introduction}
\addcontentsline{toc}{section}{Introduction}

A classical theorem by H.S.M. Coxeter states that factorizing the $n$-strand braid group $B_n$ by $p$-th order relation $\sigma ^p =1$ on its elementary braid generator $\sigma$ results in a finite quotient if and only if
\be
\lb{1}
1/n\,+\,1/p\,>\,1/2\,.
\ee
In case of $B_3$ such factorization gives  finite quotient groups for $p=2,3,4,5$, of orders, respectively, 6, 24,   96, and  600 \cite{Cox}.  Generalizing this setting one can consider quotients of the group algebra ${\Bbb C}[B_n]$
obtained by imposing $p$-th order monic polynomial relation on the elementary braids.
Under condition (\ref{1}) the resulting quotient algebras are finite dimensional and, by the Tits deformation theorem (see \cite{CurtisReiner}, $\mathsection 68$, or \cite{HR}, section 5), in a generic situation  they are isomorphic to the group algebras of the corresponding  Coxeter's quotient groups and, hence, semisimple. As a next step it would be interesting to identify the semisimplicity conditions and to describe explicitly irreducible representations of the finite dimensional quotients.

A significant progress in this direction have been achieved by I. Tuba and H. Wenzl. In paper \cite{TW} they have classified  all the irreducible representations of $B_3$ in dimensions $\leq 5$. Their classification scheme in dimensions $\leq 4$  gives all the irreducible representations for the quotients in cases $n=3$, $p=2,3,4$, and describes their semisimplicity conditions. However the  ${\Bbb C}[B_3]$ quotient algebras for $p=5$ admit irreducible representations of dimensions up to $6$ and the classification in \cite{TW} does not cover them. In this note we construct all the $6$-dimensional irreducible representations of these algebras and identify their semisimplicity conditions. We are working in the diagonal basis for the first elementary braid generator $g_1$, and we restrict our considerations to the case where all $p$ roots of its minimal polynomial are distinct. For the sake of completeness we present formulas for representations from I. Tuba and H. Wenzl list in this basis too.

Our paper is organized  as follows. In the next section we fix notations and derive preliminary results on possible values of a central element of $B_3$ in low dimensional irreducible representations ($d\leq 6$). Section 2 contains our main results: theorem \ref{theorem} ---
criteria of semisimplicity of the $p=2,3,4,5$ quotients of ${\Bbb C}[B_3]$, and proposition \ref{prop2} ---
explicit formulae for all their irreducible representations.
\medskip

Before going on with the considerations let us describe in brief related approaches and results.

In \cite{W} B. Westbury suggested approach to representation theory of $B_3$ using representations of a particular quiver. It was subsequently used by L. Le Bruyn to construct Zariski dense rational parameterizations of the irreducible representations of $B_3$ of any dimension \cite{Bruyn1,Bruyn2}. This approach has proved to be effective in treating a problem of braid reversion (see \cite{Bruyn1}). However  it does not provide representation's semisimplicity criteria. A $5$-dimensional variety of the irreducible $6$-dimensional representations of $B_3$ constructed below belongs to a $8$-dimensional family of $B_3$-representations of type 6b (see Fig.1 in \cite{Bruyn1}).

For a more general case of $B_n$, $n>3$, series of irreducible representations related with the Iwahori-Hecke  ($p=2$ case) and Birman-Murakami-Wenzl algebras ($p=3$ case with additional restrictions) are well investigated (for a review, see \cite{LR}). Some other particular families of the $B_n$-representations have been found in \cite{FLSV,AK}.

In yet another line of research M. Brou\'e, C. Malle and R. Rouquier \cite{BMR1,BMR2} generalized notions of the braid group and of the Hecke algebra associated not only to Coxeter group, but to an arbitrary finite complex reflection group $W$. Their {\em generic Hecke algebra} is defined over certain polynomial ring $R={\Bbb Z}[\{u_i\}]$. Brou\'e, Malle and Rouquer conjectured that generic Hecke algebra is a free module of rank $|W|$ over its ring of definition. This conjecture by now has been verified in many  cases (see \cite{MM,M1,M2,Ch1} and references therein), although not in all. The algebras $Q_X$ (\ref{quot-alg}) we are dealing with in this work are specializations of particular generic Hecke algebras under homomorphism $R \rightarrow {\Bbb C}$ that assigns certain complex values to the variables $u_i$. Importantly for us the freeness conjecture is proved in all these particalar cases \cite{Ch1} and therefore, dimensions of the algebras $Q_X$ coincide with the cardinalities of their corresponding Coxeter groups.

\section{Braid group $B_3$ and its quotients: spectrum of elementary braids}
\addcontentsline{toc}{section}{Braid group $B_3$ and its quotients}

The three strings braid group $B_3$ is generated by a pair of {\em elementary braids} -- $g_1$ and $g_2$ -- satisfying the braid relation
\be
\lb{braid}
g_1 g_2 g_1 = g_2 g_1 g_2 .
\ee
Alternatively it can be given in terms of generators
\be
\lb{gen-ab}
a=g_1 g_2,\quad b=g_1 g_2 g_1 ,
\ee
and relations
\be
\lb{braid2}
a^3 = b^2 = c\, ,
\ee
where $c= (g_1 g_2)^3= (g_1 g_2 g_1)^2$ is a central element of $B_3$ which generates the center ${\Bbb Z}(B_3)$. Thus,  the quotient group $B_3/{\Bbb Z}(B_3) = \langle a,b|\, a^3=b^2=1\rangle$ is the free product of two cyclic groups ${\Bbb Z}_3 * {\Bbb Z}_2$ which is known to be isomorphic to $PSL(2,{\Bbb Z})$.
\smallskip

Let  $X$ be a set of pairwise different nonzero complex numbers:
\be
\lb{X}
X = \{x_1,x_2,\dots ,x_n\}, \qquad x_i \in {\Bbb C}\setminus \{0\}, \;\; x_i\neq x_j \; \forall i\neq j.
\ee
In this note we consider finite dimensional quotient algebras of the group algebra ${\Bbb C}[B_3]$ obtained by imposing following polynomial conditions on the elementary braids:\footnote{In the braid group elementary braids $g_1$ and $g_2$ are conjugate to each other and, hence, conditions on them are identical.}
\be
\lb{poly-rel}
P_{X}(g) = \prod_{i=1}^{n=|X| }(g- x_i 1) = 0, \qquad \mbox{where $g$ is either $g_1$, or $g_2$.}
\ee
As was already mentioned in the introduction the quotient algebras
\be
\lb{quot-alg}
Q_{X} = {\Bbb C}[B_3]/\langle P_{X}(g)\rangle .
\ee
are finite dimensional iff $|X|=n<6$.  With a particular choice of polynomials $P_{X}(g)=g^n-1$ they are the group algebras of the quotient groups $B_3/\langle g^n \rangle$ and, by the Tits deformation argument,
$Q_{X} \simeq {\Bbb C}[B_3/\langle g^n\rangle]$  for  $n<6$ and for generic choice of $x_i\in X$
and, therefore, in a generic situation $Q_{X}$ is semisimple.

In the next section we will classify irreducible representations of these algebras. It turns out that their dimensions are less or equal to $6$. In the rest of this section we will show that in these irreducible  representations the spectra  of the central element $c$ (\ref{braid}) and of  generators $a$ and $b$ (\ref{gen-ab}) are, up to a discrete factor, defined by the eigenvalues $x_i$ of the elementary braids.
\smallskip

Let $V$ be a finite dimensional linear space,  $\dim V = d$.
Let $\rho_{X,V}$ be a family of irreducible representations $Q_{X} \rightarrow {\rm End}(V)$. We will assume that their characters
are continuous functions of parameters $x_i\in X$.\footnote{All representations constructed in the next section satisfy the continuity condition.}
Throughout this section  we also assume that $d \geq n$ and that the minimal polynomials of operators $\rho_{X,V}(g_{1,2})$ coincide with $P_{X}$. The latter assumptions do not cause any loss of generality since {\bf a)} all roots of the characteristic polynomials of  $\rho_{X,V}(g_{1,2})$ belong to $X$, and {\bf b)} given a family $\rho_{X',V}$ we can treat it as a family of representations of the quotient algebras $Q_{X}$ of a minimal possible set $X\subset X'$ removing from $X'$ all the elements which do not show up in the characteristic polynomials of $\rho_{X',V}(g_{1,2})$.
The characteristic polynomial of elementary braids $g_{1,2}$ in representation $\rho_{X,V}$ then has a form
\be
\lb{charpol}
\Pi_{\rho}(g) =\prod_{i=1}^{n=|X|} (g- x_i )^{m_i}, \qquad \mbox{where}\;\; m_i\in \mathbb{N}^+:\; \sum_{i=1}^n m_i =d.
\ee
In particular, $\det \rho_{X,V}(g_{1,2}) = \prod_{i=1}^n x_i^{m_i}$.
\smallskip

Denote
\be
\lb{ABC}
A:=\rho_{X,V}(a), \quad B:=\rho_{X,V}(b), \quad \rho_{X,V}(c):=C_{\rho}\, {\rm Id}_V.
\ee
Here we have taken into account that, by Schur's lemma, central element $c$ acts in the irreducible representation as a scalar operator.
Calculating determinant of $\rho_{X,V}(c)$ one finds relation
\be
\lb{det}
\left({\textstyle \prod_{i=1}^n x_i^{m_i}}\right)^6 = \left(C_{\rho}\right)^d .
_{\phantom{A_{A_{A}}}}
\ee
By (\ref{braid2}) operators $A$ and $B$ satisfy equalities
\be
\lb{AB}
A^3 = B^2 = C_{\rho}\, {\rm Id}_V.
\ee

Notice that $A$ and $B$ can not be scalar, otherwise $\rho_{X,V}(g_1)$ and $\rho_{X,V}(g_2)$ have common basis of eigenvectors and the representation $\rho_{X,V}$ is reducible. Thus, $A$ and $B$ should have at least two  different eigenvalues taking values in sets
\be
\lb{spec}
\mbox{Spec} A \subset C_{\rho}^{1/3}\cdot\{1,\nu,\nu^{-1}\},\;\; \nu:=e^{2\pi \mathrm{i} /3},\qquad
\mbox{Spec} B \subset C_{\rho}^{1/2}\cdot\{1,-1 \}\, .
\ee
The following proposition describes explicitly the spectrum of operators $A$ and $B$ in low dimensional representations.

\begin{prop}
\lb{prop1}
Let $\rho_{X,V}: Q_X \rightarrow {\rm End}(V)$ be a family of irreducible representations of algebras $Q_{X}$ (\ref{quot-alg}) such that\\
{ a)~} their characters are continuous functions of parameters $x_i\in X$;\\
{b)~} characteristic and minimal polynomials of the elementary braids $\rho_{X,V}(g_{1,2})$ are given, respectively, by
$\Pi_{\rho}$ (\ref{charpol}) and $P_{X}$ (\ref{poly-rel}).

Let $A$, $B$, $C_{\rho}$  be as defined in (\ref{ABC}). Denote $\nu :=e^{2\pi\mathrm{i}/3}$, and introduce notation $e_{k}({ X})$ for $k$-th elementary symmetric polynomial in the set of variables $X=\{x_i\}_{i=1,\dots,n}$.

Then for $n=|X|\leq 5$ and $d=\dim V\leq 6$ coefficient $C_{\rho}$ and eigenvalues of operators $A$ and $B$ can take following values.
\ba
\nn
\phantom{a}\hspace{-10mm}
d=n=2: &\hspace{-4mm}& C_{\rho}= -e_{2}({X})^3,\\
\lb{d=n=2}
\phantom{a}\hspace{-20mm}
&\hspace{-10mm}&
Spec \,A =  -e_{2}({X}) \!\cdot\! \{\nu,\nu ^{-1}\}, \quad
Spec \,B = {\mathrm i}\, e_{2}({X})^{3\over 2}\!\cdot\! \{1, -1 \};
\\[5pt]
\nn
\phantom{a}\hspace{-20mm}
d=n=3:&\hspace{-10mm}&  C_{\rho}= e_{3}({X})^2,\\
\lb{d=n=3}
\phantom{a}\hspace{-20mm}
&\hspace{-10mm}&
Spec \,A =  e_{3}({X})^{2\over 3}\!\cdot\! \{1,\nu,\nu ^{-1}\}, \quad
Spec \,B = e_{3}({X})\!\cdot\! \{1, -1^{\sharp 2}\};
\\[5pt]
\nn
\phantom{a}\hspace{-20mm}
d=n=4: &\hspace{-10mm}&  \mbox{for any root $h({X}):=\sqrt[2]{e_{4}({X})}$:}\quad
C_{\rho}= h({X})^{3},\\
\lb{d=n=4}
\phantom{a}\hspace{-20mm}
&\hspace{-10mm}&
Spec \,A = h({X})\!\cdot\!  \{1^{\sharp 2},\nu,\nu ^{-1}\}, \quad
Spec \,B = h({X})^{3\over 2}\!\cdot\! \{1^{\sharp 2}, -1^{\sharp 2}\};
\\[5pt]
\nn
\phantom{a}\hspace{-20mm}
d=n=5: &\hspace{-10mm}& \mbox{for any root $f({X}):=\sqrt[5]{e_{5}({X})}$:}\quad C_{\rho}= f({X})^{6},\\
\lb{d=n=5}
\phantom{a}\hspace{-20mm}
&\hspace{-10mm}&
Spec \,A = f({X})^{2}\!\cdot\! \{1,\nu^{\sharp 2},(\nu ^{-1})^{\sharp 2}\}, \quad
Spec \,B = f({X})^{3}\!\cdot\! \{1^{\sharp 3}, -1^{\sharp 2}\};
\ea
\ba
\nn
&
\phantom{a}\hspace{-54mm}
d=6, \; n=5,\;  m_i=2,\; 1\leq i \leq 5:\quad\;  C_{\rho}= - x_i e_{5}({X}), &
\\
\lb{d=6,n=5}
&
\phantom{a}\hspace{-10mm}
Spec \,A = -\sqrt[3]{x_i e_5({X})}\! \cdot\! \{1^{\sharp 2},\nu^{\sharp 2},(\nu ^{-1})^{\sharp 2}\}, \quad
Spec \,B = {\mathrm i} \sqrt[2]{x_i e_5({X})}\!\cdot\! \{1^{\sharp 3}, -1^{\sharp 3}\}.
&
\ea
\end{prop}
\smallskip

\begin{proof} Denote $\mbox{\rm Tr}_V$ an operation of taking trace in representation $\rho_{X,V}$. To prove assertions of the proposition we analyze functions
$\mbox{\rm Tr}_V (g_1^k g_2)$, for $k=2,\dots ,5$.

Case $d=n=2$.~ Using minimal polynomial for $g_1$ and characteristic polynomial for $g_2$ we calculate
$$
\mbox{Tr} B\, =\, \mbox{Tr}_V \left( g_1 g_2 g_1\right)\, =\, \mbox{Tr}_V \left( g_1^2 g_2\right)\, =\, e_1({X}) \left( \mbox{Tr}A - e_2({X})\right).
$$
Noticing that spectral condition (\ref{spec}) for the non-scalar $2\!\times\!2$ matrix $B$ assumes
$\mbox{Tr}B =0$ we conclude that $\mbox{Tr} A = e_2({X})$. From (\ref{det}) we have $C_{\rho}=\pm e_2({\mathrm x})^3$, which together with spectral condition on $A$ (\ref{spec}) leaves us the only possibility to fulfill relations for the traces of $A$ and $B$, namely the one presented in (\ref{d=n=2}).

Case $d=n=3$.~ We shall evaluate $\mbox{Tr}_V g_1^3 g_2$ in two different ways. First, we use cyclic property of the trace and the braid relation (\ref{braid}):
\be
\lb{trA^2}
\mbox{Tr}_V g_1^3 g_2\, =\, \mbox{Tr}_V g_1^2 g_2 g_1\, =\, \mbox{Tr}_V (g_1 g_2)^2\, =\,
\mbox{Tr} A^2.
\ee
Second, we apply minimal polynomial for $g_1$ and characteristic polynomial for $g_2$:
$$
\mbox{Tr}_V g_1^3 g_2\, =\, e_1({X}) \mbox{Tr} B\, -\, e_2({X}) \mbox{Tr} A\, +\,
e_3({X}) e_1({X}).
$$
Comparing the results of these calculations and taking into account that, by (\ref{spec}) and (\ref{det}), traces of powers of $A$ and $B$ can be expressed in terms of (roots of) $e_3({X})$ and, hence, are algebraically independent from $e_{1}({X})$ and $e_{2}({X})$ we find that
$ \mbox{Tr} A = \mbox{Tr} A^2 = 0$, $\mbox{Tr} B= - e_3({X})$.
On the other hand from (\ref{det}) one finds $C_{\rho}= \sqrt[3]{1}\, e_3({X})^2$ which, together with the spectral conditions (\ref{spec}), gives (\ref{d=n=3}) as the only possibility to satisfy the above relations for traces.

Case $d=n=4$.~
Similarly to the case $d=n=3$ we calculate $\mbox{Tr}_V g_1^4 g_2$ in two ways:
\ba
\lb{1^4 2}
&&\mbox{ Tr}_V g_1^4 g_2\, =\, \mbox{ Tr}_V (g_1 g_2)^2 g_1\, =\, C_{\rho} \mbox{ Tr}_V (g_1 g_2)^{-1} g_1\, =\, C_{\rho}\,
{e_3({X})/ e_4({X})},
\\[4pt]
\nn
&&\mbox{ Tr}_V g_1^4 g_2\, =\, e_1({X})\mbox{Tr} A^2\,-\, e_2({X})\mbox{Tr} B\, +\, e_3({X})\mbox{Tr} A
\, -\, e_4({X}) e_1({X}),
\ea
where in the last line we take additionally into account eq.(\ref{trA^2}). Hence, using an algebraic independence of
$C_{\rho}$ and thus of $\mbox{Tr} A$, $\mbox{Tr} A^2$ and $\mbox{Tr} B$ from the elementary symmetric polynomials $e_{i}({X})$, $i=1,2,3,$
one concludes: $\mbox{Tr} A = C_{\rho}/ e_4({X}),\;\; \mbox{Tr} A^2 = e_4({X}), \;\; \mbox{Tr} B=0$.
The latter conditions are only compatible with eqs.(\ref{det}) and (\ref{spec}) in two cases given in (\ref{d=n=4}).

Case $d=n=5$.~ Here we calculate $\mbox{Tr}_V g_1^5 g_2$:
\ba
\nn
\mbox{ Tr}_V g_1^5 g_2 &\!\!=\!\!& C_{\rho}\mbox{ Tr}_V (g_1 g_2)^{-1} g_1^2\, =\, C_{\rho} \mbox{ Tr}_V g_1^{-1} g_2
\\[2pt]
\nn
= &\!\!\!\!& \hspace{-6mm}
{C_{\rho}\over e_5({X})}\left( C_{\rho} {e_4({X})\over e_5({X})}\, -\, e_1({X})\mbox{Tr}A^2\, +\,
e_2({X})\mbox{Tr}B\, -\, e_3({X})\mbox{Tr}A\, +\, e_4({X}) e_1({X})\right),
\ea
where passing to the second line we expressed $g_1^{-1}$ in terms of positive powers of $g_1$ using its minimal polynomial and then used $d=5$ analogue of formula (\ref{1^4 2}).

Calculating $\mbox{Tr}_V g_1^5 g_2$ in another way we obtain
$$
\mbox{ Tr}_V g_1^5 g_2\, =\, e_1({X})\left( C_{\rho} {e_4({X})\over e_5({X})}\right)\, -\, e_2({X})\mbox{Tr} A^2\,+\, e_3({X})\mbox{Tr} B\, -\, e_4({X})\mbox{Tr} A\, +\, e_5({X}) e_1({X}).
$$
Now collecting coefficients in the independent polynomials $e_i({X}))$, $i=1,2,3,4$, and taking into account eq.(\ref{det}) we find $C_{\rho}= e_5({X})^{6/5}$,
$\mbox{Tr} A = - e_5({X})^{2/5}$, $\mbox{Tr} A^2 = - e_5({X})^{4/5}$, $\mbox{Tr} B = e_5({X})^{3/5}$, which in combination with (\ref{spec}) finally leads to conditions (\ref{d=n=5}).

Case $d=6$, $n=5$:~ We calculate $\mbox{Tr}_V g_1^5 g_2$ in two ways similarly to the previous case, but using now different
expressions $\mbox{Tr}_V g_1 =e_1({X})+x_i$, $\mbox{Tr}_V g_1^{-1}= e_4({X})/e_5({X})+x_i^{-1}$, following from the characteristic polynomial (\ref{charpol}). Collecting then coefficients in independent polynomials we derive
$C_{\rho}=-x_i e_5({X})$, $\mbox{Tr} A = \mbox{Tr} A^2 = \mbox{Tr} B = 0$, which in combination with (\ref{spec})
assumes (\ref{d=6,n=5}).
\end{proof}

\section{Low dimensional representations of $Q_X$ and semisimplicity}

In this section we construct explicitly representations of  algebras $Q_X$ whose data coincide with those given in the proposition \ref{prop1}.
Investigating reducibility conditions for these representations we obtain semisimplicity criteria for algebras $Q_X$ and classify their irreducible representations.
We derive formulas for the representations in the basis of eigenvectors of $g_1$.
\begin{prop}
\lb{prop2}
Algebras $Q_X$ in cases $|X|\leq 5$  have following representations of dimensions $\dim V\leq 6 $.
\medskip

\ni
$\;\bullet\;$  \underline{$|X|=\dim V =1:$}
\be
\lb{rep-dim1}
\rho_X^{(1)}(g_1)=\rho_X^{(1)}(g_2)=x_1 .
\ee

\ni
$\;\bullet\;$  \underline{$|X|=\dim V =2:$}
\be
\lb{rep-dim2}
\rho_{X}^{(2)}(g_1)=diag\{x_1, x_2\}, \qquad
\rho_{X}^{(2)}(g_2)=
\frac{1}{x_1-x_2}
\left(\begin{array}{cc}
-x_2^2& -x_1 x_2\\[2pt]
x_1^2-x_1x_2+x_2^2& x_1^2
\end{array}\right)\! .
\ee

\ni
$\;\bullet\;$  \underline{$|X|=\dim V =3:$}
\be
\lb{rep-dim3}
\rho_{X}^{(3)}(g_1)=diag\{x_1, x_2, x_3\}, \;\;
\rho_{X}^{(3)}(g_2)\!=\!
\left(\begin{array}{ccc}
\frac{x_2x_3(x_2+x_3)}{\Delta_1(X)}&\frac{x_3(x_1^2+x_2x_3)}{\Delta_1(X)}&\frac{x_2(x_1^2+x_2x_3)}
{\Delta_1(X)}\\[5pt]
\frac{x_3(x_2^2+x_1x_3)}{\Delta_2(X)}&\frac{x_1x_3(x_1+x_3)}{\Delta_2(X)}&\frac{x_1(x_2^2+x_1x_3)}
{\Delta_2(X)}\\[5pt]
\frac{x_2(x_3^2+x_1x_2)}{\Delta_3(X)}&\frac{x_1(x_3^2+x_1x_2)}{\Delta_3(X)}&\frac{x_1x_2(x_1+x_2)}{\Delta_3(X)}
\end{array}\right)\! ,
\ee\vspace{-6mm}

\ni
where we introduced notation
\be
\lb{Delta}
\Delta_i(X):={\prod_{j=1,\;\; j\neq i}^{|X|}}(x_j-x_i).
\ee

\ni
$\;\bullet\;$  \underline{$|X|=\dim V =4.$}\; There exist two inequivalent representations depending on a choice of the
square root $h=\sqrt{e_4(X)}$:
\ba
\nn
\rho_{h,X}^{(4)}(g_1)& =& \mbox{diag}\{x_1,x_2,x_3,x_4\} ,
\\[5pt]
\lb{rep-dim4}
\rho_{h,X}^{(4)}(g_2)&=&
\left(\begin{array}{cccc}
\frac{\alpha_1}{\Delta_1(X)}&
\frac{\beta_{1}\, \gamma_{3}\, \gamma_{4}}{\Delta_1(X)}&
\frac{\beta_{1}\, \gamma_{2}\, \gamma_{4}}{\Delta_1(X)}&
\frac{\beta_{1}\, \gamma_{2}\, \gamma_{3}}{\Delta_1(X)}
\\[6pt]
\frac{\beta_{2}}{\Delta_2(X)}&
\frac{\alpha_2}{\Delta_2(X)}&
\frac{\beta_{2}\, \gamma_{2}}{\Delta_2(X)}
&\frac{\beta_{2}\, \gamma_{2}}{\Delta_2(X)}
\\[6pt]
\frac{\beta_{3}}{\Delta_3(X)}&
\frac{\beta_{3}\, \gamma_{3}}{\Delta_3(X)}
&\frac{\alpha_3}{\Delta_3(X)}&
\frac{\beta_{3}\, \gamma_{3}}{\Delta_3(X)}
\\[6pt]
\frac{\beta_{4}}{\Delta_4(X)}&
\frac{\beta_{4}\, \gamma_{4}}{\Delta_4(X)}
&\frac{\beta_{4}\, \gamma_{4}}{\Delta_4(X)}
&\frac{\alpha_4}{\Delta_4(X)}
\end{array}\right).
\ea
\ba
\nn
\hspace{-4mm}
\mbox{Here}\;\;\;\alpha_i(h,X) &:=& e_3(X^{\setminus i})\, e_1(X^{\setminus i})\, -\, h\, e_2(X^{\setminus i}), \quad  X^{\setminus i}:= X\setminus \{x_i\},
\\
\lb{abc}
\beta_i(h,X) &:=& {e_4(X)}/{x_i^2}\, -\, h, \quad i=1,2,3,4,
\\
\nn
\gamma_a(h,X) & := & x_1 x_a\, +\, x_b x_c\, -\, h,\quad \mbox{ $a,b,c\in \{x_2,x_3,x_4\}$ are pairwise distinct}.
\ea
\vspace{0mm}

\ni
$\;\bullet\;$  \underline{$|X|=\dim V =5.$}\; There exist five inequivalent representations corresponding to different values of the
 root $f({X}):=\sqrt[5]{e_{5}({X})}$:
\ba
\lb{rep-dim5}
\rho_{f,X}^{(5)}(g_1)& =& \mbox{diag}\{x_1,x_2,x_3,x_4,x_5\},
\quad
\rho_{f,X}^{(5)}(g_2)\, =\, || m_{ij}  ||_{i,j=1}^{\quad\;\, 5},
\\[5pt]
\lb{diag-dim5}
m_{ii}(f,X)& :=& \frac{e_4(X^{\setminus i})\, e_1(X^{\setminus i}) + f\, x_i\, e_3(X^{\setminus i})+
f \prod_{k =1,\, k\neq i}^{\quad 5}(f+x_k)}{\Delta_i(X)},
\\
\lb{offdiag-dim5}
m_{ij}(f,X)& :=& \frac{(x_i^2 + f\, x_i + f^2)\prod_{k=1,\, k\neq i,j}^{\quad 5}(f^2+ x_i x_k)}
{f\, x_i\, x_j\,\Delta_i(X)}, \;\; \forall i\neq j.
\ea
\vspace{0mm}

\ni
$\;\bullet\;$  \underline{$|X|=5,\,\dim V =6.$}\; There exist  five inequivalent representations $\rho^{(6)}_{i,X},\; i=1,\dots ,5,$ corresponding to all admissible values $C_{\rho}=-x_i e_5(X)$ of the central element $c$.
Formulas for $\rho^{(6)}_{5,X}$ are given in table 1. Formulas for the other representations can be obtained by the transposition of the eigenvalues $x_5$ and $x_i$: $\rho^{(6)}_{i,X} =  \sigma_{i 5}\circ \rho^{(6)}_{5,X}$, $i=1\dots 4$.

\begin{table}[t]
\caption{6-dimensional representation of $Q_X$, $|X|=5$.}
\hspace{-6mm}
\mbox{
\begin{tabular}{|c|p{110mm}|}
\hline
\multicolumn{2}{|c|}{\rule{0pt}{20pt}
$\rho^{(6)}_{5,X}(g_1)\, =\, \mbox{diag}\{x_1,x_2,x_3, x_4, x_5, x_5\}\, ,
\qquad
\rho^{(6)}_{5,X}(g_2)\, =\, ||g_{ij}||_{i,j=1}^{\quad\;\,6},
$
}
\\[\medskipamount]
\hline\hline
\;\;$G\!:=\!||g_{ij}||_{i,j=1}^{\quad\;\,4}$\;{\rm :} \rule{0pt}{20pt}
&
\hspace{1pt}
$
g_{ii} = \frac{e_4(X^{ \setminus i}) e_1(X^{\setminus i})-x_i x_5 e_3(X^{\setminus i})}{\Delta_i(X)}, \;
X^{\setminus i}:=X\setminus \{x_i\},\;{\scriptstyle  i=1,\dots, 4};
$
\par
\rule{0pt}{15pt}
$
g_{1a}\, =\, \frac{p_a \,q_b\, q_c}{x_1^2 \Delta_a(X)}, \quad  g_{a 1}\, =\, \frac{p_1}{x_a^2 \Delta_1(X)},\quad
g_{ab}\, =\, \frac{q_{a}\, p_b}{x_a^2 \Delta_b(X)},
$
\hfill
\par
\rule{0pt}{15pt}
{\footnotesize
where indices $a, b, c \in \{2,3,4\} $ are pairwise distinct,  and}\hfill
\par
\rule{0pt}{15pt}
$
q_{a}(X) :=x_1 x_a+x_b x_c\, , \quad p_i(X):=  e_5(X)-x_i^3x_5^2\, ;
$
\hfill
\\[\medskipamount]
\hline
\rule{0pt}{22pt}
$
G_{31}\!:=\!
\left(\!\!\begin{array}{cc}
g_{51}& g_{52}\\
g_{61}& g_{62}
\end{array}\!\!\right)
$\!\!
{\rm :}
 & \par
 \hspace{1pt}
 $\mbox{diag}\{\frac{1}{\Delta_1(X)},\frac{1}{\Delta_2(X)}\}$;
\\[\bigskipamount]
\hline
\rule{0pt}{22pt}
$
G_{32}\!:=\!

\left(\!\!\begin{array}{cc}
g_{53}& g_{54}\\
g_{63}& g_{64}
\end{array}\!\!\right)
$\!\!
{\rm :}
&
$
\left(\!\!\!\begin{array}{cc}
q_4\, r& q_3\, (\sigma_{34}\!\circ\! r)\\
(\sigma_{12}\!\circ\! r) & (\sigma_{12}\sigma_{34}\!\circ\! r)
\end{array}\!\!\!\right)
$,
{\footnotesize where}
$r(X):=\frac{x_3}{x_1(x_2-x_1)\Delta_3(X^{\setminus 2})}$,\hfill
\par
\rule{0pt}{15pt}
{\footnotesize and}
$\forall f(X):\; \sigma_{ij}\!\circ\! f(\dots x_i \dots x_j \dots):=f(\ldots x_j \ldots x_i \ldots)$;\hfill
\\[\medskipamount]
\hline
\rule{0pt}{22pt}
$
G_{33}\!:=\!

\left(\!\!\begin{array}{cc}
g_{55}& g_{56}\\
g_{65}& g_{66}
\end{array}\!\!\right)
$\!\!
{\rm :}
 &
 $
\left(\!\!\!\begin{array}{cc}
u& q_3\, q_4\, v \\
(\sigma_{12}\!\circ\! v) & (\sigma_{12}\!\circ\! u)
\end{array}\!\!\!\right)
$,
{\footnotesize where}
$v(X):=\frac{p_2(X)}{x_1x_5 (x_2-x_1)\Delta_5(X^{\setminus 2})}$,\hfill
\par
\rule{0pt}{15pt}
{\footnotesize and}\;
$u(X):=\frac{x_1 x_2(x_3+x_4)(x_3 x_4-x_1 x_5)\,+\,x_3 x_4 (x_2-x_1)(x_1^2+x_2 x_5)}{(x_2-x_1)\Delta_5(X^{\setminus 2})}$;\hfill
\\[\bigskipamount]
\hline
\rule{0pt}{25pt}
$
G_{23}\!:=\!
\left(\!\!\begin{array}{cc}
g_{35}& g_{36}\\
g_{45}& g_{46}
\end{array}\!\!\right)
$\!\!
{\rm :}
 &
 $
\frac{1}{x_5 \Delta_5(X)}\left(\!\!\!\begin{array}{cc}
\frac{w}{x_3^2 }& \frac{q_3\,  (\sigma_{12}\circ w)}{x_3^2 } \\[4pt]
\frac{(\sigma_{34}\circ w)}{x_4^2 } & \frac{q_4\, (\sigma_{12}\sigma_{34}\circ w)}{x_4^2 }
\end{array}\!\!\!\right)
$,
 \par
 \rule{0pt}{15pt}
$\small w(X):={\scriptstyle  p_1(X)\bigl(x_1 x_2 x_3 x_4\{x_1 x_3+x_5(x_2+x_4)\}\,-\,x_5^3\{x_1 x_3(x_2+x_4)+x_5 x_2 x_4\}\bigr)}$;\hfill
\\[\medskipamount]
\hline
\rule{0pt}{25pt}
$
G_{13}\!:=\!
\left(\!\!\begin{array}{cc}
g_{15}& g_{16}\\
g_{25}& g_{26}
\end{array}\!\!\right)
$\!\!
{\rm :}
 &
 $
\frac{1}{\Delta_5(X)}\left(\!\!\!\begin{array}{cc}
\frac{z}{x_1}& \frac{q_3\, q_4\,  (\sigma_{12}\sigma_{23}\circ w)}{x_1^2 x_5} \\[4pt]
\frac{(\sigma_{23}\circ w)}{x_2^2 x_5} & \frac{(\sigma_{12}\circ z)}{x_2}
\end{array}\!\!\!\right)
$,
\par
\rule{0pt}{15pt}
$z(X):={\scriptstyle (e_1 e_3-x_1^2 e_2)(x_1 e_1 e_3-e_2 x_5^3) x_1 x_5\;+}$
\par
\rule{0pt}{12pt}
$\quad\;{\scriptstyle e_3 (x_1-x_5)\bigl( x_1^2 (e_1-x_1)\{e_3(x_1-x_5)-e_1 x_5^3\}\,+\, (x_1 e_2-e_3)\{x_1 e_2 +(x_1-x_5)x_5^2\}x_5\bigr)} $,
\par
\rule{0pt}{12pt}
{\footnotesize where $e_i$ are elementary symmetric polynomials in variables $x_2, x_3, x_4$}.
\\[\medskipamount]
\hline
\end{tabular}
}
\end{table}
\end{prop}

\ni {\bf Remark 1.}
{\small\rm
As it is noticed in section 1  a representation of $Q_X$ stays also a representation of $Q_{X'}$ if $X\subset X'$.
}
\vspace{2mm}

\ni {\bf Remark 2.}
{\small\rm
Irreducible representations of $B_3$ of dimensions $d \leq 5$ were classified by Imre Tuba and Hans Wenzl  in  \cite{TW}.
We reproduce their table of representations  in the basis where $g_1$ takes a diagonal form.
In their approach I.Tuba and H.Wenzl have used different basis in which matrices of the braids $g_1$ and $g_2$ assume a special `ordered' triangular  from. This allows them analyzing also algebras whose minimal polynomials $P_X$ have multiple roots and, hence, matrices of the braids $g_{1,2}$ are not diagonalizable. These cases are missed in our approach. Instead, our method is suitable for construction
of the 6-dimensional representations for algebras $Q_X$, $|X|=5$ and, thus, allows us classifying irreducible representations for these algebras and studying their semisimplicity.

Note also that formulas for representations of dimensions $d \leq 5$ were reconstructed in \cite{Ch1} using different methods
with the help of the CHEVIE package of GAP3 (see \cite{MM,Mich}).
}

\begin{proof}
By our initial assumptions  matrices of braids $g_{1,2}$ in any representation are diagonalizable.
We choose a basis where $\rho_{X,V}(g_1) := D_g$ is diagonal. By (\ref{charpol}) the diagonal components of $D_g$ are $x_i$ taken with multiplicities $m_i$.

Keeping in mind that in an irreducible representation matrices
$A$ and $B$ of braids $a$ and $b$ are also diagonalizable (see eq.(\ref{AB})) we use  for them parameterization
\be
\lb{AB-diag}
A = U^{-1} D_a U, \quad B= V D_b V^{-1}.
\ee
Here $D_a$ and $D_b$ are diagonal matrices whose diagonal components are elements of $Spec\,A$ and $Spec\,B$.
For irreducible representations of dimensions $\leq 6$ they were defined in proposition \ref{prop1}.
Due to relation  $g_1=a^{-1}b$ matrices $U$ and $V$ have to satisfy condition
\be
\lb{UV-cond}
U\, D_g\, V\, =\, D_a^{-1} \, U\, V\, D_b.
\ee
We solve this matrix equality for $U$ and $V$ in cases where diagonal matrices $D_g$, $D_a$ and $D_b$ are as described in proposition \ref{prop1}. Formulae for representations given in proposition \ref{prop2} follow then, e.g., from relation $g_2=g_1^{-1} a$:
$
\;\rho_{X,V}(g_2)= D_g^{-1} A.
$\smallskip

Solving (\ref{UV-cond}) is straightforward but rather tedious computation.
For an interested reader we give few details of it in
cases $d=2,3,4$.\medskip

Case $d=2$. We choose
$$
D_g=diag\{x_1,x_2\},\quad D_a=-e_2(X)\, diag\{\nu,\nu^{-1}\}, \quad D_b= {\rm i} {e_2(X)}^{3\over 2}\,diag\{1,-1\}.
$$
Noticing that matrices $U$/$V$ are defined up to left/right multiplication by a diagonal matrix we use for them
following ansatzes
$$
U=\left(\!\!\begin{array}{cc}
1& *\\
*& 1
\end{array}\!\!\right) , \quad
V=\left(\!\!\begin{array}{cc}
1& *\\
*& 1
\end{array}\!\!\right) ,
$$
where stars stay for unknown components. With this settings eq.(\ref{UV-cond}) defines $U$ and $V$ up to conjugation by a diagonal matrix.
We choose a solution which gives nice expression (\ref{rep-dim2}) for $\rho_{X}^{(2)}(g_2)$:
$$
U=\left(\!\!\begin{array}{cc}
1& -\frac{x_1}{ \nu^{-1}x_1 + \nu x_2}\\[2pt]
-\frac{\nu x_1 +\nu^{-1} x_2}{x_1}& 1
\end{array}\!\!\right) , \quad
V=\left(\!\!\begin{array}{cc}
1& -\frac{{\rm i}\sqrt{e_2}}{x_1-x_2+{\rm i}\sqrt{e_2}}\\[2pt]
\frac{x_1-x_2-{\rm i}\sqrt{e_2}}{{\rm i}\sqrt{e_2}}& 1
\end{array}\!\!\right) ,
$$
Note that, unlike $U$ and $V$, resulting expression for $\rho_X^{(2)}(g_2)$ is defined with the only restriction  $x_1\neq x_2$
and does not depend on a choice of root $\sqrt{e_2}$.
\medskip

Case $d=3$. We choose
$$
D_g=diag\{x_1,x_2,x_3\},\;\; D_a={e_3(X)}^{2\over 3}\, diag\{1,\nu^{-1},\nu\}, \;\; D_b= e_3(X)\,diag\{1,-1,-1\},
$$
and use ansatzes
$$
U=\left(\!\!\begin{array}{ccc}
1& *&* \\
*& 1& *\\
*& *& 1
\end{array}\!\!\right) , \quad
V=\left(\!\!\begin{array}{ccc}
1& *&* \\
*& 1& 0\\
*& 0& 1
\end{array}\!\!\right) .
$$
Solution of eq.(\ref{UV-cond}) which gives formula (\ref{rep-dim3}) for $\rho_X^{(3)}(g_2)$ reads
$$
U=\left(\!\!\begin{array}{ccc}
1& {x_1+h\over x_2+h}&{x_1+h\over x_3+h} \\[3pt]
{x_2+\nu h\over x_1 + \nu h}& 1& {x_2+\nu h\over x_3+\nu h}\\[3pt]
{x_3+\nu^{-1} h\over x_1+\nu^{-1} h}& {x_3+\nu^{-1} h\over x_2+\nu^{-1} h}& 1
\end{array}\!\!\right) , \quad
V=\left(\!\!\begin{array}{ccc}
1& -1&-1 \\[2pt]
-{(x_1-x_3)(x_2^2+x_1 x_3)\over (x_2-x_3)(x_1^2+x_2 x_3)}& 1& 0\\[5pt]
-{(x_1-x_2)(x_3^2+x_1 x_2)\over (x_3-x_2)(x_1^2+x_2 x_3)}& 0& 1
\end{array}\!\!\right) .
$$
\medskip

Case $d=4$. We choose $\;\;D_g=diag\{x_1,x_2,x_3,x_4\},$
$$
D_a=h(X)\, diag\{1,1,\nu,\nu^{-1}\}, \quad D_b= {h(X)}^{3\over 2}\,diag\{1,1,-1,-1\},
$$
and  ansatzes for $U$, $V$:
$$
U=\left(\!\!\begin{array}{cc}
I& \Psi^+\\
\Psi^-& \Phi
\end{array}\!\!\right) , \quad
V=\left(\!\!\begin{array}{cc}
I& \Lambda^+\\
\Lambda^-& I
\end{array}\!\!\right) ,
$$
where $I$ is $2\times 2$ unit matrix, $\Phi^\pm$ and $\Lambda^\pm$ are arbitrary $2\times 2$ matrices, and $2\times 2$ matrix $\Phi$ has unit diagonal components. Particular solution of eq.(\ref{UV-cond}) which gives expression (\ref{rep-dim4}) for $\rho_{h,X}^{(4)}(g_2)$ reads
\ba
\nn
\Psi^+ = \left(\!\!\begin{array}{cc}
\frac{x_1(x_3-x_2)\beta_1\gamma_4}{x_3(x_1-x_2)\beta_3}&
\frac{x_1(x_4-x_2)\beta_1\gamma_3}{x_4(x_1-x_2)\beta_4}
\\[4pt]
\frac{x_2(x_3-x_1)\beta_2}{x_3(x_2-x_1)\beta_3}&
\frac{x_2(x_4-x_1)\beta_2}{x_4(x_2-x_1)\beta_4}
\end{array}\!\!\right) ,
& &
\Psi^- = \left(\!\!\begin{array}{cc}
\frac{x_1 x_2}{(x_1 x_2 +\nu^{-1} h)(x_2 x_3 +\nu h)}&
\frac{x_2 x_4 + \nu h}{x_3 x_4 +\nu h}
\\[4pt]
\frac{x_1 x_2}{(x_1 x_2 +\nu h)(x_2 x_4 +\nu^{-1} h)}&
\frac{x_2 x_3 + \nu^{-1} h}{x_3 x_4 +\nu^{-1} h}
\end{array}\!\!\right) ,
\ea
\ba
\nn
\Phi &=& \left(\!\!\begin{array}{cc}
1&
\frac{x_2 x_4+ \nu h}{x_2 x_3+ \nu h}
\\[4pt]
\frac{x_2 x_3+ \nu^{-1} h}{x_2 x_4+ \nu^{-1} h}&
1
\end{array}\!\!\right) ,
\\[4pt]
\nn
\Lambda^+ &=& -\left(\!\!\begin{array}{cc}
\frac{x_3(x_3-x_2)(x_1-\sqrt{h})\gamma_4}{x_1(x_1-x_2)(x_3-\sqrt{h})}&
\frac{x_4(x_4-x_2)(x_1-\sqrt{h})\gamma_3}{x_1(x_1-x_2)(x_4-\sqrt{h})}
\\[4pt]
\frac{x_3(x_3-x_1)(x_2-\sqrt{h})}{x_2(x_2-x_1)(x_3-\sqrt{h})}&
\frac{x_4(x_4-x_1)(x_2-\sqrt{h})}{x_2(x_2-x_1)(x_4-\sqrt{h})}
\end{array}\!\!\right) ,
\\[4pt]
\nn
\Lambda^- &=& -\frac{1}{\gamma_2}\left(\!\!\begin{array}{cc}
\frac{x_1(x_4-x_1)(x_3+\sqrt{h})}{x_3(x_4-x_3)(x_1+\sqrt{h})}&
\frac{x_2(x_4-x_2)(x_3+\sqrt{h})\gamma_3}{x_3(x_4-x_3)(x_2+\sqrt{h})}
\\[4pt]
\frac{x_1(x_3-x_1)(x_4+\sqrt{h})}{x_4(x_3-x_4)(x_1+\sqrt{h})}&
\frac{x_2(x_3-x_2)(x_4+\sqrt{h})\gamma_4}{x_4(x_3-x_4)(x_2+\sqrt{h})}
\end{array}\!\!\right) .
\ea
To get it we exclude consecutively matrices $\Lambda^\pm$, $\Psi^-$, $\Phi$ from equations (\ref{UV-cond}) expressing them finally in terms of $\Psi^+$. The only condition imposed by eq.(\ref{UV-cond}) on the components of $\Psi^+$ is
$$
\frac{(\Psi^+)_{11}(\Psi^+)_{22}}{(\Psi^+)_{12}(\Psi^+)_{21}} = \frac{(x_3-x_2)(x_4-x_1)\gamma_4}{(x_4-x_2)(x_3-x_1)\gamma_3}.
$$
Remaining three degrees of freedom are due to arbitrariness in conjugation of $U$ and $V$ by a diagonal matrix. We fix it to get the expression for $\rho_X^{(4)}(g_2)$ in the most suitable form.

Solving eq.(\ref{UV-cond}) in cases $d=5$, $\dim V=5,6,$ is more lengthy. We skip it presenting final results of the calculations in eqs.(\ref{rep-dim5})-(\ref{offdiag-dim5}) and  in table 1. For them the braid relation (\ref{braid}) can be checked directly.
\end{proof}

\begin{prop}
\lb{prop3}
For algebras $Q_X$ (\ref{quot-alg}) defined by a set of data $X$ (\ref{X}) representations $\rho^{(d)}_{\dots}$, $d\leq 5$, described in proposition \ref{prop2}
are irreducible if and only if following conditions on their parameters are satisfied
\ba
\lb{irr-dim2}
|X|=2,\;\;\rho^{(2)}_X: &&
I^{(2)}_{ij}:= x_i^2-x_ix_j+x_j^2 \neq 0 ,
\\
\nn
&&\mbox{\small where indices $\scriptstyle i, j\in \{1,2\}$ are distinct;}
\\[4pt]
\lb{irr-dim3}
|X|=3,\;\; \rho^{(3)}_X: &&
I^{(3)}_{ijk}:= x_i^2 + x_jx_k  \neq 0,
\\
\nn
&&
\mbox{\small where  $\scriptstyle i,j,k\in \{1,2,3\}$ are pairwise distinct;}
\\[4pt]
\lb{irr-dim4}
|X|=4,\; \rho^{(4)}_{h,X}: &&
I^{(4)}_{h,i}:=   x_i^2- h \neq 0, \quad J^{(4)}_{h,ijkl}:= x_i x_j +x_k x_l -h \neq 0,
\\
\nn
&&
\mbox{\small where $\scriptstyle i,j,k,l\in \{1,2,3,4\}$ are pairwise distinct;}
\\[4pt]
\lb{irr-dim5}
|X|=5,\; \rho^{(5)}_{f,X}: &&
I^{(5)}_{f,i}:= x_i^2+ x_i f +f^2\neq 0, \quad
J^{(5)}_{f,ij}:=x_i x_j+f^2 \neq  0,
\\
\nn
&&
\mbox{\small where $\scriptstyle i,j\in \{1,2,3,4,5\}$ are pairwise distinct;}
\ea
Otherwise, they are reducible but indecomposable.

For representations $\rho^{(6)}_{s,X}$, $s=1,\dots ,5$, also given in proposition \ref{prop2} we present
less detailed statement, which describes conditions under which all of them are irreducible:
\ba
\nn
|X|=5,\;
\rho^{(6)}_{s,X}, {\scriptstyle 1\leq s\leq 5}:&&
I^{(6)}_{i}:= e_5(X)+x_i^5\neq 0, \quad J^{(6)}_{ij}:= e_5(X)-x_i^3 x_j^2\neq 0,
\\
\lb{irr-dim6}
&&
 K^{(6)}_{i,jklm}:= x_j x_k + x_l x_m \neq 0,
\\
\nn
 &&
\mbox{\small where $\scriptstyle i,j,k,l,m\in \{1,2,3,4,5\}$ are pairwise distinct.}
\ea
Otherwise, among them there are reducible but indecomposable representations.
\end{prop}

\begin{proof}
We will search for invariant subspaces in representations $\rho^{(d)}_{\dots}$ of proposition \ref{prop2}.
Note that for any $y\in Q_X$ such that $Spec\,\rho_{X,V}(y)$ is multiplicity free
an invariant subspace in $V$ should be a linear span of some subset of a basis of eigenvectors of $\rho_{X,V}(y)$.

Consider representations $\rho^{(d)}_{\dots}$ of dimension $d=\dim V \leq 5$. Here the spectrum of $\rho^{(d)}_{\dots}(g_1)$ is simple.
Choose a basis of eigenvectors of $\rho^{(d)}_{\dots}(g_1)$: $\{v_k:=\delta_{ki},\, 1\leq  i\leq d\}_{k=1,\dots d}$. Denote
\be
\lb{subspace}
V_{Y}:= Span\{v_k: k\in Y\}, \; \mbox{where $Y\subset \{1,2,3,4,5\}$}.
\ee
Obviously, any invariant subspace in the representation space $V$, if exists, should be of the form $V_Y$.
Furthermore, if the representation is decomposable then the decomposition is
\be
\lb{decomposition}
V=V_Y\oplus V_{\bar Y},\; \mbox{where ${\bar Y}:= \{1,2,3,4,5\}\setminus Y$}.
\ee
Correspondingly,  matrix $\rho^{(d)}_{\dots}(g_2)$ have to be block-triangular (resp., block-diagonal) with blocks labelled by indices
from subsets $Y$ and $\bar Y$, iff the representation is reducible (resp., decomposable).
Let us analyze the block structure of $\rho^{(d)}_{\dots}(g_2)$ in cases $d=3,4,5$ (case $d=2$ is trivial).

Case $d=3$. Representation $\rho^{(3)}_X$ (\ref{rep-dim3}) has 2-dimensional invariant subspace $V_{\{1,2\}}$ iff $I^{(3)}_{312}=0$.
Its complementary 1-dimensional subspace $V_{\{3\}}$ exists under conditions $I^{(3)}_{123}=I^{(3)}_{231}=0$. Altogether conditions
$I^{(3)}_{312}=I^{(3)}_{123}=I^{(3)}_{231}=0$ lead to $x_1=x_2=x_3=0$ and, hence, they are incompatible. Invariance conditions in two other cases --- $V_{\{2,3\}}$, $V_{\{1\}}$, and $V_{\{1,3\}}$, $V_{\{2\}}$ ---
differ from the above by a cyclic permutation of the subscript indices.
It follows that $\rho^{(3)}_X$ is irreducible iff inequalities (\ref{irr-dim3}) are fulfilled, and
otherwise it is indecomposable.
\medskip

Case $d=4$. Conditions for existence of invariant subspaces in $\rho^{(4)}_{h,X}$ are
\ba
\lb{123-4}
V_{\{1,2,3\}} :\; I^{(4)}_{h,4} =0;&&V_{\{4\}}:\; I^{(4)}_{h,3} =J^{(4)}_{h,1234}=0,\; \mbox{or}\;\; I^{(4)}_{h,2} =J^{(4)}_{h,1324}=0;
\\[2pt]
\lb{12-34}
V_{\{1,2\}}:\;  I^{(4)}_{h,3} =I^{(4)}_{h,4} =0; && V_{\{3,4\}}:\; J^{(4)}_{h,1234}=0,\; \mbox{or}\;\; I^{(4)}_{h,1} =I^{(4)}_{h,2} =0.
\ea
For the rest of invariant subspaces their existence conditions can be obtained by a cyclic permutations of subscripts $1,2,3,4$ in (\ref{123-4})
\footnote{The only exception is subspace $V_{\{1\}}$ which can not be invariant in this representation.}, or of subscripts $2,3,4$ in (\ref{12-34}). Altogether these conditions justify irreducibility criterium (\ref{irr-dim4}). Decomposability, e.g., like
$V=V_{\{1,2,3\}}\oplus V_{\{4\}}$, or like $V=V_{\{1,2\}}\oplus V_{\{3,4\}}$, demands
$$
 I^{(4)}_{h,1} =I^{(4)}_{h,2}=I^{(4)}_{h,3} =I^{(4)}_{h,4} =0,
\;\quad \mbox{or}\;\quad
I^{(4)}_{h,3} = I^{(4)}_{h,4} = J^{(4)}_{h,1234}=0,$$
or similar sets of relations with permuted subscripts $2,3,4$. One can check that these conditions are incompatible with initial settings for $X$ (\ref{X}).\medskip

Case $d=5$. Invariant subspaces in $\rho^{(5)}_{f,X}$ exist under conditions:
\ba
\lb{1234-5}
V_{\{1,2,3,4\}}\! :\, I^{(5)}_{f,5} =0;&&V_{\{5\}}\!:\, J^{(5)}_{f,12} =J^{(5)}_{f,34}=0,\, \mbox{\small or $\forall$ permutation  of sbs $\scriptstyle 2,3,4$, or}
\\
\nn
&&\quad J^{(5)}_{f,12} =I^{(5)}_{f,3}
=I^{(5)}_{f,4}=0,\, \mbox{\small or $\forall$ permutation of subscripts $\scriptstyle 1,2,3,4$};
\ea
\ba
\lb{123-45}
V_{\{1,2,3\}}\!:\,  J^{(5)}_{f,45} =0,\; \mbox{or}\;\; I^{(5)}_{f,4} =I^{(5)}_{f,5} =0;
&&
V_{\{4,5\}}\!:\, I^{(5)}_{f,3}=J^{(5)}_{f,12} =0,\hspace{30mm}
\\
\nn
&&
\quad \mbox{\small or $\forall$ permutation of subscripts $\scriptstyle 1,2,3$}.
\ea
For the rest of invariant subspaces the existence conditions can be obtained by permutation of indices in formulas above. Taken together these conditions prove irreducibility criterium (\ref{irr-dim5}). On the other hand, an attempt to find decomposition into invariant subspaces,
like $V=V_{\{1,2,3,4\}}\oplus V_{\{5\}}$, or like $V=V_{\{1,2,3\}}\oplus V_{\{4,5\}}$, results in a set of conditions
$$
 I^{(5)}_{f,1}\!=\!J^{(5)}_{f,23}\!=\!J^{(5)}_{f,45}\!=\!0,\,\mbox{\small or}\;
 I^{(5)}_{f,1}\!=\!I^{(5)}_{f,2}\!=\!I^{(5)}_{f,3}\!=\!J^{(5)}_{f,45}\!=\!0,\;
 \mbox{\small or $\forall$ permutation of sbs $1,2,3,4,5$},
$$
which are incompatible with (\ref{X}). Thus, representations $\rho^{(5)}_{f,X}$ are always indecomposable.
\medskip

Case $d=6$ is more sophisticated.  We carry out considerations for representation $\rho^{(6)}_{5,X}$
(see table 1).
For the other 6-dimensional representations results follow then by transpositions of arguments $x_i$.

Take a basis of eigenvectors of $\rho^{(6)}_{5,X}(g_1)$:
 $\{v_k:=\delta_{ki},\, 1\leq  i\leq 6\}_{k=1,\dots 6}$. Assume there exists an invariant subspace
 $V_{inv}\subsetneq V$ and consider its subspace
 $$
 W:=V_{inv}\cup V_{\{1,2,3 ,4\}}.
 $$
Spectrum of $\rho^{(6)}_{5,X}(g_1)$ in this subspace is simple and so, $W$ has a form $W=V_Y$  (\ref{subspace}) for some subset $Y\subset\{1,2,3,4\}$. We consider separately cases with different $Y$.
\smallskip

Case $W=V_{\{1,2,3,4\}}$. Consider action of matrix $\rho^{(6)}_{5,X}(g_2)$ on $W$.
Since components $g_{51}$ and $g_{62}$ of this matrix are always nonzero we conclude that
vectors $v_5$ and $v_6$ belong to $V_{inv}$ and hence, $V_{inv}=V$, which is a contradiction.
\smallskip

Case $W=V_{\{1\}}$. Considering action of $\rho^{(6)}_{5,X}(g_2)$ on $v_1\in W\subset V_{inv}$
we  obtain $v_5\in V_{inv}$. Now let's assume that $V_{inv}=V_{\{1,5\}}$. Then the matrix $\rho^{(6)}_{5,X}(g_2)$ should take block-diagonal form with vanishing components $g_{21}=g_{31}=g_{41}=g_{61}=g_{25}=g_{35}=g_{45}=g_{65}=0$. This happens iff $p_1(X)\equiv J^{(6)}_{15}=0$. Thus, we conclude that
representation $\rho^{(6)}_{5,X}$ under condition $J^{(6)}_{15}=0$ has the invariant subspace $V_{\{1,5\}}$. This subspace is not further reducible.
\smallskip

Case $W=V_{\{2,3\}}$. From the action of $\rho^{(6)}_{5,X}(g_2)$ on $v_2\in V_{inv}$ we get $v_6\in V_{inv}$, as $g_{26}\neq 0$. Assuming then $V_{inv}=V_{\{2,3,6\}}$ and checking block-triangularity of
$\rho^{(6)}_{5,X}(g_2)$: $g_{12}=g_{13}=g_{16}=g_{42}=g_{43}=g_{46}=g_{52}=g_{53}=g_{56}=0$,
we find that this case is realized under condition $q_4(X)\equiv K^{(6)}_{5,1423}=0$. Thus, $V_{\{2,3,6\}}$ is a minimal invariant subspace containing $W=V_{\{2,3\}}$.

Two cases considered  above illustrate appearance of conditions like $J^{(6)}_{\dots}\neq 0$ and
$K^{(6)}_{\dots}\neq 0$ in formulation of the proposition.  Permuting arguments $x_i$, that is, considering all
representations $\rho^{(6)}_{\dots}$ one can obtain all polynomials $J^{(6)}_{\dots}$, $K^{(6)}_{\dots}$
in the conditions of their reducibility.
Consideration of the other cases with $W\neq \emptyset$ is similar.
It does not result in any other independent reducibility conditions. In particular, for representation
$\rho^{(6)}_{5,X}$ one obtains:
\begin{itemize}
\item[-] in case $W=V_{\{2,3,4\}}$ minimal possible invariant  subspace $V_{inv}=V_{\{2,3,4,5,6\}}$;
\item[-] in case  $W=V_{\{1,4\}}$ minimal possible invariant subspace $V_{inv}=V_{\{1,4,5,6\}}$.
\end{itemize}
In searching for a decomposition of $\rho^{(6)}_{5,X}$ into a direct sum these invariant subspaces could be complements, respectively, for the subspaces $V_{inv}=V_{\{1,2\}}$ (case $W=V_{\{1\}}$) and $V_{inv}=V_{\{2,3,6\}}$ (case $W=V_{\{2,3\}}$). As we see, this does not happen.  In all other reducible regimes with
$W\neq \emptyset$ representations $\rho^{(6)}_{\dots}$ turn to be indecomposable.
\medskip

It lasts considering case $W=\emptyset$.
Assuming that $V_{inv}$ is 2-dimensional, i.e. $V_{inv}=V_{\{5,6\}}$, we get a contradiction since block-triangularity conditions for $\rho^{(6)}_{5,X}$: $G_{13}=G_{23}=0$ do not have any solution.

Still, there is a possibility to find 1-dimensional space $V_{inv}$. This happens if $2\times 2$ matrices $G_{13}$, $G_{23}$ and $G_{33}$ for certain values of parameters $x_i$ have common eigenspace $V_{inv}$, which is a null space for $G_{13}$ and $G_{23}$. Calculating determinants of $G_{13}$ and $G_{23}$:
$$
\det G_{13}\sim K^{(6)}_{5,1234} J^{(6)}_{35} J^{(6)}_{45} (e_5(X)+ x_5^5), \quad
\det G_{23}\sim J^{(6)}_{15} J^{(6)}_{25} (e_5(X)+ x_5^5),
$$
we see that the only new possible regime where one observes nontrivial invariant subspace is given by condition
$I^{(6)}_5=0$. Indeed, in this case one finds common eigenvector
$$
\{ (x_5^2+x_2 x_3)(x_5^2-x_1 x_3)(x_2^2-x_2 x_5+x_5^2),\, x_1 x_3(x_1^2-x_1 x_5+x_5^2)\},
$$
with eigenvalues $0$, $0$ and $x_5$, respectively, for $G_{13}$, $G_{23}$ and $G_{33}$. The invariant subspace generated by this vector does not have an invariant direct summand, as there is no invariant subspaces containing $V_{\{1,2,3,4\}}$.

\end{proof}

Our main result follows as a direct consequence of propositions \ref{prop2} and \ref{prop3}:

\begin{theor}
\lb{theorem}
For $|X|\leq 5$ algebra $Q_X$ (\ref{quot-alg}) defined by a set of data $X$ (\ref{X}) is semisimple iff:\medskip
\ba
\lb{simp-2}
|X|=2: && I^{(2)}_{12}\neq 0;
\\[5pt]
\lb{simp-3}
|X|=3: && \{I^{(2)}_{ij}, I^{(3)}_{ijk}\}\cap \{0\}=\emptyset \;\;
\mbox{\small for all pairwise distinct indices $\scriptstyle i,j,k\in \{1,2,3\}$};
\\[5pt]
\lb{simp-4}
|X|=4: && \{I^{(2)}_{ij}, I^{(3)}_{ijk}, I^{(4)}_{h,i}, J^{(4)}_{h,ijkl}\}\cap \{0\}=\emptyset
\\
\nn
&&
\forall h: h^2= e_4(X), \mbox{\small and for all pairwise distinct indices $\scriptstyle i,j,k,l\in \{1,2,3,4\}$};
\\[5pt]
\lb{simp-5}
|X|=5:  && \{I^{(2)}_{ij}, I^{(3)}_{ijk}, I^{(4)}_{h,i}, J^{(4)}_{h,ijkl},
I^{(5)}_{f,i}, J^{(5)}_{f,ij}, I^{(6)}_i, J^{(6)}_{ij}, K^{(6)}_{i,jklm}\}\cap \{0\}=\emptyset \;\;
\\
\nn
&&\forall f: f^5=e_5(X),\;\;\; \forall h: h^2= e_4(X^{\setminus i}),
\\
\nn
&&
\mbox{\small and for all pairwise distinct indices $\scriptstyle i,j,k,l,m\in \{1,2,3,4,5\}$}.
\ea

In the semisimple case all irreducible representations of these algebras are described in proposition \ref{prop2}.
\end{theor}

\ni {\bf Remark 3.}
{\small\rm
For the algebras $Q_X$, $|X|=2,3,4$, the statement of theorem was first proved in \cite{TW} (see Theorem 2.9 there).
For the algebras $Q_x$, $|X|=5$, polynomial conditions of the form $I^{(6)}_i = 0$, $J^{(6)}_{ij}=0$, $K^{(6)}_{i,jklm}=0$
have appear recently in the investigations of the algebra decomposition matrices (see \cite{Ch2}, Section 3.15).
}

\begin{proof}
Existence of reducible but indecomposable representations assumes nonsemisimplicity of an algebra. All the algebras $Q_X$ which the theorem states to be nonsemisimple obey such representations according to proposition \ref{prop3}.

On the other hand, by the Artin-Wedderburn theorem an algebra over an algebraically closed field is semisimple if and only if sum of squares of dimensions of its inequivalent irreducible representations equals dimension of the algebra.
Dimensions of the algebras $Q_X$ for $|X|=2,3,4,$ and $5$ are, respectively, 6, 24, 96, and 600 (see \cite{M1}, Theorem 3.2(3), and \cite{Ch1}, Corollaries 3.4 and 4.11).
Then, propositions \ref{prop2} and \ref{prop3} provide enough irreducible representations for algebras $Q_X$
to guarantee their semisimplicity under conditions (\ref{simp-2})--(\ref{simp-5}).
For instance in case $|X|=5$ the algebra $Q_X$  under conditions (\ref{simp-5}) has following inequivalent irreducible representations (see proposition \ref{prop2} and remark
1):
${5 \choose 1}=5$ times 1-dimensional, ${5\choose 2}=10$ times 2-dimensional, ${5\choose 3}=10$ times 3-dimensional, $2\times {5 \choose 4}=10$ times 4-dimensional, $5$ times 5-dimensional, and $5$ times 6-dimensional. Altogether:
$5*1^2+10*2^2+10*3^2+10*4^2+5*5^2+5*6^2=600$ that fits the dimension of the algebra and proves its semisimplicity.
\end{proof}

\end{document}